\documentclass[1p,fleqn,preprint]{elsarticle}
\usepackage{amsmath,amsfonts,amsthm}
\usepackage{mathrsfs}
\usepackage{ifpdf}
\ifpdf
\usepackage[bookmarks=false,%
pdfstartview=FitH,linkbordercolor={0.5 1 1},%
citebordercolor={0.5 1 0.5},unicode,%
hyperindex]{hyperref}%
\else
\usepackage{breakurl}                          
\fi

\biboptions{numbers,sort&compress}

\newtheorem{theorem}{Theorem}
\newtheorem{lemma}{Lemma}
\newtheorem{example}{Example}
\newtheorem{corollary}{Corollary}
\newtheorem*{daitheorem}{Dai's Theorem~\cite{Dai:JFI14}}

\newcommand{\setA}{\mathscr{A}}

\newcommand{\vvv}{{|\!|\!|}}

\journal{Linear Algebra and its Applications}

\begin{document}

\begin{frontmatter}

\title{The Berger-Wang formula\\
for the Markovian joint spectral radius\tnoteref{rfbr}}

\tnotetext[rfbr]{Supported by the Russian Foundation for Basic Research,
Project No. 13-01-13105.}

\author{Victor Kozyakin}

\address{Institute for Information Transmission
Problems\\ Russian Academy of Sciences\\ Bolshoj Karetny lane 19, Moscow
127994 GSP-4, Russia}

\ead{kozyakin@iitp.ru}
\ead[url]{http://www.iitp.ru/en/users/46.htm}

\begin{abstract}
The Berger-Wang formula establishes equality between the joint and
generalized spectral radii of a set of matrices. For matrix products whose
multipliers are applied not arbitrarily but in accordance with some
Markovian law, there are also known analogs of the joint and generalized
spectral radii. However, the known proofs of the Berger-Wang formula hardly
can be directly applied in the case of Markovian products of matrices since
they essentially rely on the arbitrariness of appearance of different
matrices in the related matrix products. Nevertheless, as has been shown by
X.~Dai~\cite{Dai:JFI14} the Berger-Wang formula is valid for the case of
Markovian analogs of the joint and the generalized spectral radii too,
although the proof in this case heavily exploits the more involved
techniques of multiplicative ergodic theory. In the paper we propose a
matrix theory construction allowing to deduce the Markovian analog of the
Berger-Wang formula from the classical Berger-Wang formula.
\end{abstract}


\end{frontmatter}

\section{Introduction}\label{S-intro}
Let $\mathbb{K}=\mathbb{R},\mathbb{C}$ be the field of real or complex
numbers, and $\setA=\{A_{1},A_{2},\ldots,A_{N}\}$ be a finite set of
$(d\times d)$-matrices with the elements from $\mathbb{K}$. Given a
sub-multiplicative norm\footnote{A norm $\|\cdot\|$ on a space of linear
operators is called sub-multiplicative if $\|AB\|\le\|A\|\cdot\|B\|$ for any
operators $A$ and $B$.} $\|\cdot\|$ on $\mathbb{K}^{d\times d}$, the limit
\begin{equation}\label{E-JSRad}
\rho({\setA}):=
\limsup_{n\to\infty}\rho_{n}({\setA})\qquad \left(~ =
\lim_{n\to\infty}\rho_{n}({\setA}) =
\inf_{n\ge1}\rho_{n}({\setA})\right),
\end{equation}
where
\[
\rho_{n}({\setA}):=
\sup\left\{\|A_{i_{n}}\cdots A_{i_{1}}\|^{1/n}:~i_{j}\in \{1,2,\ldots,N\}\right\},
\]
is called the \emph{joint spectral radius} of the set of matrices $\setA$
\cite{RotaStr:IM60}. This limit always exists and does not depend on the norm
$\|\cdot\|$. If $\setA$ is a singleton set then (\ref{E-JSRad}) turns into
the known Gelfand formula for the spectral radius of a linear operator. By
this reason sometimes (\ref{E-JSRad}) is called the generalized Gelfand
formula \cite{ShihWP:LAA97}.

The \emph{generalized spectral radius} of the set of matrices $\setA$ is the
quantity defined by a similar to \eqref{E-JSRad} formula in which instead of
the norm is taken the spectral radius $\rho(\cdot)$ of the corresponding
matrices \cite{DaubLag:LAA92,DaubLag:LAA01}:
\begin{equation}\label{E-GSRad}
\hat{\rho}({\setA}):=
\limsup_{n\to\infty}\hat{\rho}_{n}({\setA})\qquad \left(~ = \sup_{n\ge1}\hat{\rho}_{n}({\setA})\right),
\end{equation}
where
\[
\hat{\rho}_{n}({\setA}):=\sup\left\{\rho(A_{i_{n}}\cdots A_{i_{1}})^{1/n}:~i_{j}\in \{1,2,\ldots,N\}\right\}.
\]

As has been noted by M.~Berger and Y.~Wang~\cite{BerWang:LAA92} the
quantities $\rho({\setA})$ and $\hat{\rho}({\setA})$ for bounded sets of
matrices $\setA$ in fact coincide with each other:
\begin{equation}\label{E-bergWang}
\hat{\rho}({\setA})=\rho({\setA}).
\end{equation}
This fundamental formula has numerous applications in the theory of
joint/generalized spectral radius. In particular, it implies the continuous
dependence of the joint/general\-ized spectral radius on the set of matrices
$\setA$. Another important consequence of the \emph{Berger-Wang formula}
\eqref{E-bergWang} is the fact that the quantities $\hat{\rho}_{n}({\setA})$
and $\rho_{n}({\setA})$, for any $n$, form the lower and upper bounds
respectively for the joint/generalized spectral radius of the set of matrices
$\setA$:
\begin{equation}\label{E-GSR-JSR}
\hat{\rho}_{n}({\setA})\le
\hat{\rho}({\setA})=\rho({\setA})\le
\rho_{n}({\setA}),
\end{equation}
which may serve as the basis for estimating the accuracy of computation of
the joint/ge\-neral\-ized spectral radius.


The characteristic feature of the definitions \eqref{E-JSRad} and
\eqref{E-GSRad} is that the matrix products $A_{i_{n}}\cdots A_{i_{1}}$ in
them correspond to all the possible sequences of indices
$(i_{1},\ldots,i_{n})$. Much more complicated is the situation when the
matrix products $A_{i_{n}}\cdots A_{i_{1}}$ in formulae \eqref{E-JSRad} and
\eqref{E-GSRad} are subjected to some additional restrictions, for example,
some combinations of matrices in them are forbidden. Let us describe in more
details a situation of the kind.

Given an $(N\times N)$-matrix $\varOmega=(\omega_{ij})$ with the elements
from the binary set $\{0,1\}$ then the finite sequence $(i_{1},\ldots,i_{n})$
taking the values in $\{1,2,\ldots,N\}$ will be called
\emph{$\varOmega$-admissible} if $\omega_{i_{j+1}i_{j}}=1$ for all $1\le j
\le n-1$ and there exists $i_{*}\in\{1,2,\ldots,N\}$ such that
$\omega_{i_{*}i_{n}}=1$. Denote by $W_{N,\varOmega}$ the set of all
$\varOmega$-admissible sequences $(i_{1},\ldots,i_{n})$. The matrix products
$A_{i_{n}}\cdots A_{i_{1}}$ corresponding to the $\varOmega$-admissible
sequences $(i_{1},\ldots,i_{n})$ will be called \emph{Markovian} since the
products of matrices of the kind arise naturally in the theory of matrix
cocycles over the topological Markov chains, see, e.g.,
\cite{KatokHas:e,Kitchens98}.

Now, define analogs of formulae \eqref{E-JSRad} and \eqref{E-GSRad} for the
$\varOmega$-admissible products of matrices. The limit
\begin{equation}\label{E-JSRadM}
\rho({\setA},\varOmega):=
\limsup_{n\to\infty}\rho_{n}({\setA},\varOmega),
\end{equation}
where
\[
\rho_{n}({\setA},\varOmega):=\sup\left\{\|A_{i_{n}}\cdots A_{i_{1}}\|^{1/n}:~
(i_{1},\ldots,i_{n})\in W_{N,\varOmega}\right\},
\]
will be called the \emph{Markovian joint spectral radius} of the set of
matrices $\setA$ defined by the \emph{matrix of admissible transitions
$\varOmega$}. If, for some $n$, the set of all $\varOmega$-admissible
sequences $(i_{1},\ldots,i_{n})$ is empty then put
$\rho_{n}({\setA},\varOmega)=0$. In this case for each $k\ge n$ the sets of
all $\varOmega$-admissible sequences $(i_{1},\ldots,i_{k})$ will be also
empty, and hence $\rho({\setA},\varOmega)=0$. The question whether there
exist arbitrarily long $\varOmega$-admissible sequences can be resolved in a
finite number of steps. In particular, the set $W_{N,\varOmega}$ has
arbitrarily long sequences if each column of the matrix $\varOmega$ contains
at least one nonzero element.

The limit \eqref{E-JSRadM} always exists and does not depend on the norm
$\|\cdot\|$. To justify this, let us note that the quantity
$\rho_{n}^{n}({\setA},\varOmega)$ is sub-multiplicative in $n$. Then, like in
the case of formula \eqref{E-JSRad}, by the Fekete Lemma~\cite{Fekete:MZ23}
(see also \cite[Ch.~3, Sect.~1]{PolyaSezgoI}) there exist
$\lim_{n\to\infty}\rho_{n}({\setA},\varOmega)$ and
$\inf_{n\ge1}\rho_{n}({\setA},\varOmega)$, and both of them are equal to the
limit \eqref{E-JSRadM}:
\[
\rho({\setA},\varOmega):=
\limsup_{n\to\infty}\rho_{n}({\setA},\varOmega)=\lim_{n\to\infty}\rho_{n}({\setA},\varOmega)=
\inf_{n\ge1}\rho_{n}({\setA},\varOmega).
\]

The quantity
\begin{equation}\label{E-GSRadM}
\hat{\rho}({\setA},\varOmega):=
\limsup_{n\to\infty}\hat{\rho}_{n}({\setA},\varOmega),
\end{equation}
where
\[
\hat{\rho}_{n}({\setA},\varOmega):=\sup\left\{\rho(A_{i_{n}}\cdots A_{i_{1}})^{1/n}:~
(i_{1},\ldots,i_{n})\in W_{N,\varOmega}\right\},
\]
will be called the \emph{Markovian generalized spectral radius} of the set of
matrices $\setA$ defined by the matrix of admissible transitions $\varOmega$.
Here again we put $\hat{\rho}_{n}({\setA},\varOmega)=0$ if the set of
$\varOmega$-admissible sequences of indices $(i_{1},\ldots,i_{n})$ is empty.
Like in the case of formula \eqref{E-GSRad}, the limit \eqref{E-GSRadM}
coincides with $\sup_{n\ge1}\hat{\rho}_{n}({\setA},\varOmega)$.

For the Markovian products of matrices there are valid the inequalities
\begin{equation}\label{E-MGSR-MJSR}
\hat{\rho}_{n}({\setA},\varOmega)\le
\hat{\rho}({\setA},\varOmega)\le\rho({\setA},\varOmega)\le
\rho_{n}({\setA},\varOmega),
\end{equation}
similar to \eqref{E-GSR-JSR}. However the question whether there is valid the
equality
\begin{equation}\label{E-bergWangMark}
\hat{\rho}({\setA},\varOmega)=\rho({\setA},\varOmega),
\end{equation}
similar to the Berger-Wang equality~\eqref{E-bergWang}, becomes more
complicated. The reason is that the known proofs
\cite{BerWang:LAA92,Els:LAA95,ShihWP:LAA97,Bochi:LAA03,Dai:JMAA11} of the
classical Berger-Wang formula~\eqref{E-bergWang} essentially use the fact
that different matrices in the related matrix products can be multiplied in
an arbitrary order. Impossibility to multiply matrices in an arbitrary order,
in the Markovian case, requires to develop a different approach. The arising
difficulties have been overcome by X.~Dai in~\cite{Dai:JFI14} by using the
techniques of the multiplicative ergodic theory. To formulate the related
assertion we need some auxiliary definitions.

An $\varOmega$-admissible finite sequence $(i_{1},\ldots,i_{n})$ will be
referred to as \emph{periodically extendable} if $\omega_{i_{1}i_{n}}=1$. In
general, not every $\varOmega$-admissible finite sequence can be periodically
extended. However, if there are arbitrarily long $\varOmega$-admissible
sequences then there exist also arbitrarily long $\varOmega$-admissible
periodically extendable sequences. The set of all $\varOmega$-admissible
periodically extendable sequences will be denoted by
$W^{(\text{per})}_{N,\varOmega}$.

Define the quantity
\[
\hat{\rho}^{(\text{per})}_{n}({\setA},\varOmega):=\sup\left\{\rho(A_{i_{n}}\cdots A_{i_{1}})^{1/n}:~
(i_{1},\ldots,i_{n})\in W^{(\text{per})}_{N,\varOmega}\right\},
\]
and set\footnote{Like in the definitions of the Markovian joint and
generalized spectral radii we put
$\hat{\rho}^{(\text{per})}_{n}({\setA},\varOmega)=0$ if the set of all the
periodically extendable sequences of length $n$ is empty.}
\begin{equation}\label{E-GSRadMp}
\hat{\rho}^{(\text{per})}({\setA},\varOmega):=
\limsup_{n\to\infty}\hat{\rho}^{(\text{per})}_{n}({\setA},\varOmega).
\end{equation}

\begin{daitheorem}
$\hat{\rho}^{(\text{per})}({\setA},\varOmega)=\rho({\setA},\varOmega)$.
\end{daitheorem}

Since $W^{(\text{per})}_{N,\varOmega}\subseteq W_{N,\varOmega}$ then
$\hat{\rho}^{(\text{per})}_{n}({\setA},\varOmega)\le
\hat{\rho}_{n}({\setA},\varOmega)$ for each $n\ge 1$, and therefore
$\hat{\rho}^{(\text{per})}({\setA},\varOmega)\le
\hat{\rho}({\setA},\varOmega)$. This last inequality together with
\eqref{E-MGSR-MJSR} by Dai's Theorem then implies the Markovian
analog~\eqref{E-bergWangMark} of the Berger-Wang formula~\eqref{E-bergWang}.

The goal of the paper is to propose a matrix theory approach, called for
brevity in the next section the ``$\varOmega$-lift of the set of matrices
$\setA$'', which will allow to reduce consideration of the Markovian products
of matrices to consideration of arbitrary (all possible) products of some
auxiliary matrices. Thereby we will be able deduce the Markovian analog of
the Berger-Wang formula from the classical Berger-Wang formula.

\section{$\boldsymbol{\varOmega}$-lift of the set of matrices $\setA$}\label{S-Bernprod}

Recall that $\setA=\{A_{1},A_{2},\ldots,A_{N}\}$ is a finite set of $(d\times
d)$-matrices with the elements from the field
$\mathbb{K}=\mathbb{R},\mathbb{C}$ of real or complex numbers, and
$\varOmega=(\omega_{ij})$ is an $(N\times N)$-matrix with the elements $0$ or
$1$ which defines admissibility of consecutive co-multipliers in the products
of matrices from $\setA$.

For each $i=1,2,\ldots,N$ define the $(N\times N)$-matrix
\begin{equation}\label{E-Omegai} \varOmega_{i}=
\boldsymbol{\omega}_{i}^{T}\boldsymbol{\delta}_{i},
\end{equation}
where the row-vectors $\boldsymbol{\omega}_{i}$ and $\boldsymbol{\delta}_{i}$
are of the form
\begin{equation}\label{E-defei}
\boldsymbol{\omega}_{i}=\{\omega_{1i},\omega_{2i},\ldots,\omega_{Ni}\},\quad
\boldsymbol{\delta}_{i}=\{\delta_{1i},\delta_{2i},\ldots,\delta_{Ni}\},
\end{equation}
with $\delta_{ij}$ being the Kronecker symbol, and the upper index ``$T$''
denotes transposition of a vector. Then all the nonzero elements of the
matrix $\varOmega_{i}$, if any, belong to the $i$-th column of the matrix,
and what is more, these elements coincide with the corresponding elements of
the matrix $\varOmega$.

Now, define the set $\setA_{\varOmega}\subset\mathbb{K}^{Nd\times Nd}$ of
block $(N\times N)$-matrices $A^{(i)}:=\varOmega_{i}\otimes A_{i}$ whose
elements are $(d\times d)$-matrices, i.e.,
\[
\setA_{\varOmega}:=\left\{\varOmega_{i}\otimes A_{i}: i=1,2,\ldots,N\right\},
\]
where $\otimes$ stands for the Kronecker products of matrices~\cite{HJ:e}.
The set of matrices $\setA_{\varOmega}$ will be called the
\emph{$\varOmega$-lift of the set of matrices $\setA$}.

\begin{example}\label{Ex1}
Let $N=4$ and
\[
\varOmega=\left(
\begin{array}{cccc}
1&0&0&1\\
0&0&1&1\\
1&1&0&1\\
0&1&1&0\\
\end{array}\right).
\]
Then
\begin{align*} \varOmega_{1}&=\left(
\begin{array}{cccc}
1&0&0&0\\
0&0&0&0\\
1&0&0&0\\
0&0&0&0\\
\end{array}\right),\quad
\varOmega_{2}=\left(
\begin{array}{cccc}
0&0&0&0\\
0&0&0&0\\
0&1&0&0\\
0&1&0&0\\
\end{array}\right),\\
\varOmega_{3}&=\left(
\begin{array}{cccc}
0&0&0&0\\
0&0&1&0\\
0&0&0&0\\
0&0&1&0\\
\end{array}\right),\quad
\varOmega_{4}=\left(
\begin{array}{cccc}
0&0&0&1\\
0&0&0&1\\
0&0&0&1\\
0&0&0&0\\
\end{array}\right)
\end{align*}
and
\begin{align*} A^{(1)}&=\left(
\begin{array}{cccc}
A_{1}&0&0&0\\
0&0&0&0\\
A_{1}&0&0&0\\
0&0&0&0\\
\end{array}\right),\quad
A^{(2)}=\left(
\begin{array}{cccc}
0&0&0&0\\
0&0&0&0\\
0&A_{2}&0&0\\
0&A_{2}&0&0\\
\end{array}\right),\\
A^{(3)}&=\left(
\begin{array}{cccc}
0&0&0&0\\
0&0&A_{3}&0\\
0&0&0&0\\
0&0&A_{3}&0\\
\end{array}\right),\quad
A^{(4)}=\left(
\begin{array}{cccc}
0&0&0&A_{4}\\
0&0&0&A_{4}\\
0&0&0&A_{4}\\
0&0&0&0\\
\end{array}\right).
\end{align*}
\end{example}

Given a sub-multiplicative norm $\|\cdot\|$ on the space of $(d\times
d)$-matrices $\mathbb{K}^{d\times d}$, define the norm $\vvv\cdot\vvv$ on
$\mathbb{K}^{Nd\times Nd}$ by setting, for a block $(N\times N)$-matrix
$M=(m_{ij})$ with the elements $m_{ij}\in\mathbb{K}^{d\times d}$,
\[
\vvv M\vvv:=\max_{1\le i\le N}\sum_{j=1}^{N}\|m_{ij}\|.
\]
The norm $\vvv\cdot\vvv$ is also sub-multiplicative which is seen from the
following inequalities for block matrices $B=(b_{ij})$ and $C=(c_{ij})$:
\begin{align*}
\vvv BC\vvv&= \max_{1\le i\le N}\sum_{j=1}^{N}\left\|\sum_{k=1}^{N}b_{ik}c_{kj}\right\|\le
\max_{1\le i\le N}\sum_{j=1}^{N}\sum_{k=1}^{N}\|b_{ik}\|\|c_{kj}\|\\
&=\max_{1\le i\le N}\sum_{k=1}^{N}\sum_{j=1}^{N}\|b_{ik}\|\|c_{kj}\|
=\max_{1\le i\le N}\sum_{k=1}^{N}\left(\|b_{ik}\|\sum_{j=1}^{N}\|c_{kj}\|\right)\\
&\le\max_{1\le i\le N}\sum_{k=1}^{N}\left(\|b_{ik}\|\cdot\vvv C\vvv\right)
=\left(\max_{1\le i\le N}\sum_{k=1}^{N}\|b_{ik}\|\right)\cdot\vvv C\vvv
=\vvv B\vvv\cdot\vvv C\vvv.
\end{align*}

\begin{theorem}\label{T-main} The quantities
\begin{align*}
\rho_{n}(\setA_{\varOmega})&:=\sup\left\{\vvv A^{(i_{n})}\cdots A^{(i_{1})}\vvv^{1/n}:~A^{(i_{k})}\in\setA_{\varOmega}\right\},\\
\hat{\rho}_{n}(\setA_{\varOmega})&:=\sup\left\{\rho(A^{(i_{n})}\cdots A^{(i_{1})})^{1/n}:~A^{(i_{k})}\in\setA_{\varOmega}\right\},\\
\rho_{n}({\setA},\varOmega)&:=\sup\left\{\|A_{i_{n}}\cdots A_{i_{1}}\|^{1/n}:~
(i_{1},\ldots,i_{n})\in W_{N,\varOmega}\right\},\\
\hat{\rho}^{(\text{per})}_{n}({\setA},\varOmega)&:=\sup\left\{\rho(A_{i_{n}}\cdots A_{i_{1}})^{1/n}:~
(i_{1},\ldots,i_{n})\in W^{(\text{per})}_{N,\varOmega}\right\},
\end{align*}
for each $n\ge1$ satisfy the equalities
\[
\rho_{n}(\setA_{\varOmega})=\rho_{n}({\setA},\varOmega),\qquad
\hat{\rho}_{n}(\setA_{\varOmega})=\hat{\rho}^{(\text{per})}_{n}({\setA},\varOmega).
\]
\end{theorem}

This theorem together with the Berger-Wang formula~\eqref{E-bergWang} implies
the claim of Dai's Theorem:
$\hat{\rho}^{(\text{per})}({\setA},\varOmega)=\rho({\setA},\varOmega)$. As
was mentioned by X.~Dai in a private communication to the author, from
Theorem~\ref{T-main} it follows also the Lipschitz continuity of the
Markovian joint spectral radius with respect to $\setA$. Specifically, there
is valid the following Markovian analog of Wirth's Theorem
from~\cite{Wirth:LAA02}.
\begin{corollary}\label{Cor4-main}
The map $\setA\mapsto\rho({\setA},\varOmega)$ is locally Lipschitz continuous
on the set of all $\setA$ for which the set of matrices $\setA_{\varOmega}$
is irreducible\footnote{A set of matrices is called irreducible if all the
matrices from this set do not have common invariant subspaces except the
trivial zero space and the whole space.}.
\end{corollary}

To prove Corollary~\ref{Cor4-main} it suffices to note that the classical
joint spectral radius is locally Lipschitz continuous on the variety of all
the irreducible matrix sets~\cite{Wirth:LAA02,Koz:LAA10}, and the map
$\setA\mapsto\setA_{\varOmega}$ is also Lipschitz continuous. This, by
Theorem~\ref{T-main}, implies the claim of Corollary~\ref{Cor4-main}.

Let us remark that irreducibility of the set of matrices $\setA_{\varOmega}$
depends not only on irreducibility of $\setA$ but also on the structure of
the matrix of admissible transitions $\varOmega$. In general, neither
irreducibility of $\setA_{\varOmega}$ follows from irreducibility of $\setA$
nor irreducibility of $\setA_{\varOmega}$ implies irreducibility of $\setA$.

\section{Proof of Theorem~\ref{T-main}}

To prove Theorem~\ref{T-main} we need the following
\begin{lemma}\label{L-prodT}
Given matrices $\varOmega_{i_{k}}$, $k=1,2,\ldots,n$, then
\begin{enumerate}
\item[(i)]it is valid the representation:
\begin{equation}\label{E-Omegaprod}
\varOmega_{i_{n}}\varOmega_{i_{n-1}}\cdots \varOmega_{i_{1}}=
(\omega_{i_{n}i_{n-1}}\cdots \omega_{i_{2}i_{1}})\cdot
\boldsymbol{\omega}_{i_{n}}^{T}\boldsymbol{\delta}_{i_{1}},
\end{equation}
where the vectors $\boldsymbol{\omega}_{i}$ and $\boldsymbol{\delta}_{i}$
are of the form~\eqref{E-defei};

\item[(ii)]$\varOmega_{i_{n}}\varOmega_{i_{n-1}}\cdots
    \varOmega_{i_{1}}\neq0$ if and only if the sequence
    $(i_{1},\ldots,i_{n})$ is $\varOmega$-admissible, i.e.,
\begin{equation}\label{E-allowprod}
\omega_{i_{k+1}i_{k}}=1,\quad k=1,2,\ldots,n-1,
\end{equation}
    and there exists $i_{*}$ such that $\omega_{i_{*}i_{n}}=1$. In this
    case only the $(i_{*},i_{1})$-th elements of the matrix
    $\varOmega_{i_{n}}\varOmega_{i_{n-1}}\cdots \varOmega_{i_{1}}$
    satisfying $\omega_{i_{*}i_{n}}=1$ are other than zero;

\item[(iii)]if condition~\eqref{E-allowprod} holds then the matrix
    $\varOmega_{i_{n}}\varOmega_{i_{n-1}}\cdots \varOmega_{i_{1}}$ has a
    diagonal nonzero element if and only if the sequence
   $(i_{1},\ldots,i_{n})$ is  $\varOmega$-admissible and periodically
   extendable, i.e., $\omega_{i_{1}i_{n}}=1$. In this case the nonzero
   diagonal element of the matrix
   $\varOmega_{i_{n}}\varOmega_{i_{n-1}}\cdots \varOmega_{i_{1}}$ belongs
   to the $i_{1}$-th row and column.
\end{enumerate}
\end{lemma}

\begin{proof}
By using representations \eqref{E-Omegai} for the matrices
$\varOmega_{i_{k}}$ we get
\[
\varOmega_{i_{n}}\varOmega_{i_{n-1}}\cdots \varOmega_{i_{1}}=
\boldsymbol{\omega}_{i_{n}}^{T}\boldsymbol{\delta}_{i_{n}} \boldsymbol{\omega}_{i_{n-1}}^{T}\boldsymbol{\delta}_{i_{n-1}}
\cdots \boldsymbol{\omega}_{i_{1}}^{T}\boldsymbol{\delta}_{i_{1}},
\]
where the vectors $\boldsymbol{\omega}_{i_{k}}$ and
$\boldsymbol{\delta}_{i_{k}}$ are defined by equalities \eqref{E-defei}. Here
$\boldsymbol{\delta}_{i_{k+1}}\boldsymbol{\omega}_{i_{k}}^{T}=\omega_{i_{k+1}i_{k}}$
for each $k=1,2,\ldots,n-1$ from which equality~\eqref{E-Omegaprod} follows.
Assertion (i) is proved.

By \eqref{E-Omegaprod} $\varOmega_{i_{n}}\varOmega_{i_{n-1}}\cdots
\varOmega_{i_{1}}\neq0$ if and only if condition \eqref{E-allowprod} holds
and besides
$\boldsymbol{\omega}_{i_{n}}^{T}\boldsymbol{\delta}_{i_{1}}\neq0$. But
$\boldsymbol{\omega}_{i_{n}}^{T}\boldsymbol{\delta}_{i_{1}}\neq0$ if and only
if there exists $i_{*}$ such that $\omega_{i_{*}i_{n}}=1$. From here
assertion (ii) follows.

At last, if condition \eqref{E-allowprod} holds then from assertion~(ii) it
follows that the matrix $\varOmega_{i_{n}}\varOmega_{i_{n-1}}\cdots
\varOmega_{i_{1}}$ has a diagonal nonzero element if and only if
$\omega_{i_{1}i_{n}}=1$. This last equality together with
condition~\eqref{E-allowprod} means that the sequence of indices
$(i_{1},\ldots,i_{n})$ is not only $\varOmega$-admissible but also
periodically extendable. Assertion (iii) is proved.
\end{proof}

Let us illustrate the statement of Lemma~\ref{L-prodT} by an example.
\begin{example}\label{Ex2}
Let $\varOmega$ be the same matrix as in Example~\ref{Ex1}. Then
$\varOmega_{4}\varOmega_{2}\varOmega_{1}=0$ because of $\omega_{21}=0$ and
therefore the sequence of indices $(1,2,4)$ is not $\varOmega$-admissible.

At the same time $\omega_{31}=1$, $\omega_{43}=1$, $\omega_{14}=1$. Then the
sequence of indices $(1,3,4)$ is $\varOmega$-admissible and periodically
extendable. Hence by Lemma~\ref{L-prodT}
$\varOmega_{4}\varOmega_{3}\varOmega_{1}\neq0$ and the diagonal element of
the matrix $\varOmega_{4}\varOmega_{3}\varOmega_{1}$ belonging to the $1$-st
row and column is other than zero. This conclusion is supported by the
following expression:
\[
\varOmega_{4}\varOmega_{3}\varOmega_{1}=\left(
\begin{array}{cccc}
1&0&0&0\\
1&0&0&0\\
1&0&0&0\\
0&0&0&0\\
\end{array}\right).
\]
\end{example}

Now, start proving Theorem~\ref{T-main}. Fix an $n$, and for a sequence of
indices $(i_{1},\ldots,i_{n})$ consider the matrix
$A^{(i_{n})}A^{(i_{n-1})}\cdots A^{(i_{1})}$. By using the identity
\[
(P\otimes Q)(R\otimes S)\equiv (PR)\otimes (QS)
\]
which, in particular, holds for any matrices $P,R$ of dimensions $N\times N$
and matrices $Q,S$ of dimensions $d\times d$ (see \cite[Lemma~ 4.2.10]{HJ:e})
we obtain
\begin{equation}\label{E-Aijprod}
A^{(i_{n})}A^{(i_{n-1})}\cdots
A^{(i_{1})}= \varOmega_{i_{n}}\varOmega_{i_{n-1}}\cdots
\varOmega_{i_{1}} \otimes A_{i_{n}}A_{i_{n-1}}\cdots
A_{i_{1}}.
\end{equation}

First prove the inequality
\begin{equation}\label{E-rhorho1}
\rho_{n}(\setA_{\varOmega})\le\rho_{n}({\setA},\varOmega).
\end{equation}
Let us note that the products of matrices $A^{(i_{k})}$ for which
$A^{(i_{n})}A^{(i_{n-1})}\cdots A^{(i_{1})}=0$ do not contribute to the
computation of the quantity $\rho_{n}(\setA_{\varOmega})$. Therefore it
suffices to consider the case when $A^{(i_{n})}A^{(i_{n-1})}\cdots
A^{(i_{1})}\neq0$. By \eqref{E-Aijprod} the latter may happen only in the
case when $\varOmega_{i_{n}}\varOmega_{i_{n-1}}\cdots
\varOmega_{i_{1}}\neq0$. But, by assertions (i) and (ii) of
Lemma~\ref{L-prodT}, $\varOmega_{i_{n}}\varOmega_{i_{n-1}}\cdots
\varOmega_{i_{1}}\neq0$ if and only if the sequence of indices
$(i_{1},\ldots,i_{n})$ is $\varOmega$-admissible, that is,
$(i_{1},\ldots,i_{n})\in W_{N,\varOmega}$. In this case equality
\eqref{E-Aijprod} by assertion (i) of Lemma~\ref{L-prodT} can be rewritten in
the form:
\begin{equation}\label{E-Aijprod1}
A^{(i_{n})}A^{(i_{n-1})}\cdots
A^{(i_{1})}= (\boldsymbol{\omega}_{i_{n}}^{T}\boldsymbol{\delta}_{i_{1}})\otimes A_{i_{n}}A_{i_{n-1}}\cdots
A_{i_{1}},
\end{equation}
where $(\boldsymbol{\omega}_{i_{n}}^{T}\boldsymbol{\delta}_{i_{1}})\neq0$.
So, in the right-hand part of \eqref{E-Aijprod1} stands a block matrix all
the block elements of which belong to a single column and coincide with
$A_{i_{n}}A_{i_{n-1}}\cdots A_{i_{1}}$. From here and from the definition of
the norm $\vvv\cdot\vvv$ it follows that
\[
\vvv A^{(i_{n})}A^{(i_{n-1})}\cdots
A^{(i_{1})}\vvv^{1/n}
= \|A_{i_{n}}A_{i_{n-1}}\cdots
A_{i_{1}}\|^{1/n}\le \rho_{n}({\setA},\varOmega).
\]
These last relations hold for any set of block matrices
$A^{(i_{k})}\in\setA_{\varOmega}$, $k=1,2,\ldots,n$, satisfying $\vvv
A^{(i_{n})}A^{(i_{n-1})}\cdots A^{(i_{1})}\vvv\neq0$, and therefore they
imply \eqref{E-rhorho1}.

Prove the inequality reciprocal to \eqref{E-rhorho1}. As has been already
proved the product of matrices $A^{(i_{k})}$ can be represented in the form
\eqref{E-Aijprod}. Then by assertions (i) and (ii) of Lemma~\ref{L-prodT} for
an $\varOmega$-admissible sequence $(i_{1},\ldots,i_{n})$ (i.e., such that
$(i_{1},\ldots,i_{n})\in W_{N,\varOmega}$) holds equality \eqref{E-Aijprod1},
where $(\boldsymbol{\omega}_{i_{n}}^{T}\boldsymbol{\delta}_{i_{1}})\neq0$.
Then
\[
\|A_{i_{n}}A_{i_{n-1}}\cdots
A_{i_{1}}\|^{1/n}= \vvv A^{(i_{n})}A^{(i_{n-1})}\cdots
A^{(i_{1})}\vvv^{1/n}\le \rho_{n}(\setA_{\varOmega}).
\]
But since these last relations are valid for any $\varOmega$-admissible
sequence $(i_{1},\ldots,i_{n})$ then
\begin{equation}\label{E-rhorho2}
\rho_{n}({\setA},\varOmega) \le \rho_{n}(\setA_{\varOmega}).
\end{equation}
From \eqref{E-rhorho1} and \eqref{E-rhorho2} we obtain the first claim
Theorem~\ref{T-main}:
$\rho_{n}(\setA_{\varOmega})=\rho_{n}({\setA},\varOmega)$.

To prove the second claim of Theorem~\ref{T-main} let us observe that by
assertion~(i) of Lemma~\ref{L-prodT} the block matrix in the right-hand part
of \eqref{E-Aijprod} may have nonzero elements only in one (block) column. In
this case its spectral radius may be nonzero only in the case when it has
nonzero diagonal element. By assertion~(iii) of Lemma~\ref{L-prodT} the
latter may happen if and only if the sequence $(i_{1},\ldots,i_{n})$ is
$\varOmega$-admissible and periodically extendable, that is,
$(i_{1},\ldots,i_{n})\in W^{(\text{per})}_{N,\varOmega}$. By
\eqref{E-Aijprod1} in this case it is valid the following equality
\begin{equation}\label{E-Aijsprad}
\rho(A^{(i_{n})}A^{(i_{n-1})}\cdots
A^{(i_{1})})
= \rho(A_{i_{n}}A_{i_{n-1}}\cdots
A_{i_{1}}).
\end{equation}

Let us note that equality \eqref{E-Aijsprad} is valid for any set of block
matrices $A^{(i_{k})}\in\setA_{\varOmega}$ satisfying $\vvv
A^{(i_{n})}A^{(i_{n-1})}\cdots A^{(i_{1})}\vvv\neq0$ as well as for any
sequence $(i_{1},\ldots,i_{n})\in W_{N,\varOmega}$ and the corresponding set
of matrices $A^{(i_{k})}\in\setA_{\varOmega}$. From here, like in the case of
the first claim of the theorem, we get
\[
\hat{\rho}_{n}(\setA_{\varOmega})=\hat{\rho}_{n}^{(\text{per})}({\setA},\varOmega).
\]

Theorem~\ref{T-main} is proved.\qed

\section{Concluding remarks}\label{S-comments}

In this section we shall discuss alternative definitions of the Markovian
joint spectral radius, and also the possibility to apply the techniques of
$\varOmega$-lifts of the set of matrices $\setA$ in the cases not mentioned
above.

\subsection{Alternative definitions of the Markovian joint spectral radius}
The definition \eqref{E-JSRadM} of the Markovian joint spectral radius is
essentially depends on the fact over which sets of matrices the norms of the
products of matrices are maximized during computing the quantities
$\rho_{n}({\setA},\varOmega)$, i.e., on how is the notion of
$\varOmega$-admissible sequences defined.

For example, we could treat a finite sequence $(i_{1},\ldots,i_{n})$
$\varOmega$-admissible if $\omega_{i_{j+1}i_{j}}=1$ was carried out for all
$1\le j \le n-1$ (not assuming existence of $i_{*}$ such that
$\omega_{i_{*}i_{n}}=1$). The set of all the finite sequences
$(i_{1},\ldots,i_{n})$, $\varOmega$-admissible in this sense, will be denoted
by $W^{(0)}_{N,\varOmega}$.

Also, we could treat a finite sequence $(i_{1},\ldots,i_{n})$
$\varOmega$-admissible if it was a starting interval of some infinite to the
right sequence $(i_{1},\ldots,i_{n},\ldots)$ for which the relations
$\omega_{i_{j+1}i_{j}}=1$ were valid for all $j\ge 1$. The set of all the
finite sequences $(i_{1},\ldots,i_{n})$, $\varOmega$-admissible in this
sense, will be denoted by $W^{(\infty)}_{N,\varOmega}$.

Clearly,
\[
W^{(\text{per})}_{N,\varOmega}\subseteq  W^{(\infty)}_{N,\varOmega}\subseteq
W_{N,\varOmega}\subseteq W^{(0)}_{N,\varOmega}.
\]
Then
\begin{equation}\label{E-rhoineq}
\rho^{(\text{per})}_{n}({\setA},\varOmega)\le  \rho^{(\infty)}_{n}({\setA},\varOmega)\le
\rho_{n}({\setA},\varOmega)\le \rho^{(0)}_{n}({\setA},\varOmega),\qquad n\ge 1,
\end{equation}
where
\begin{align*}
\rho^{(\text{per})}_{n}({\setA},\varOmega)&:=\sup\left\{\|A_{i_{n}}\cdots A_{i_{1}}\|^{1/n}:~
(i_{1},\ldots,i_{n})\in W^{(\text{per})}_{N,\varOmega}\right\},\\
\rho^{(\infty)}_{n}({\setA},\varOmega)&:=\sup\left\{\|A_{i_{n}}\cdots A_{i_{1}}\|^{1/n}:~
(i_{1},\ldots,i_{n})\in W^{(\infty)}_{N,\varOmega}\right\},\\
\rho^{(0)}_{n}({\setA},\varOmega)&:=\sup\left\{\|A_{i_{n}}\cdots A_{i_{1}}\|^{1/n}:~
(i_{1},\ldots,i_{n})\in W^{(0)}_{N,\varOmega}\right\},\\
\intertext{and the quantity $\rho_{n}({\setA},\varOmega)$ has been defined already as}
\rho_{n}({\setA},\varOmega)&:=\sup\left\{\|A_{i_{n}}\cdots A_{i_{1}}\|^{1/n}:~
(i_{1},\ldots,i_{n})\in W_{N,\varOmega}\right\}.
\end{align*}

Observe that $\|A_{i_{n}}A_{i_{n-1}}\cdots A_{i_{1}}\|\le \alpha
\|A_{i_{n-1}}\cdots A_{i_{1}}\|$ with $\alpha=\max_{1\le i\le N} \|A_{i}\|$,
and besides $(i_{1},\ldots,i_{n-1})\in W_{N,\varOmega}$ for any
$(i_{1},\ldots,i_{n-1},i_{n})\in W^{(0)}_{N,\varOmega}$. Then
\begin{multline*}
\sup\left\{\|A_{i_{n}}\cdots A_{i_{1}}\|:~
(i_{1},\ldots,i_{n})\in W^{(0)}_{N,\varOmega}\right\}\\
\le\alpha \sup\left\{\|A_{i_{n-1}}\cdots A_{i_{1}}\|:~
(i_{1},\ldots,i_{n-1})\in W_{N,\varOmega}\right\},\qquad n>1,
\end{multline*}
and therefore
\begin{equation}\label{E-rhonnm1}
\rho^{(0)}_{n}({\setA},\varOmega)\le
\alpha^{1/n} \left(\rho_{n-1}({\setA},\varOmega)\right)^{(n-1)/n},\qquad n>1.
\end{equation}

From \eqref{E-rhoineq} and \eqref{E-rhonnm1}, and from the evident inequality
$\hat{\rho}^{(\text{per})}_{n}({\setA},\varOmega)\le
\rho^{(\text{per})}_{n}({\setA},\varOmega)$ it follows that
\[
\hat{\rho}^{(\text{per})}_{n}({\setA},\varOmega)
\le\rho^{(\text{per})}_{n}({\setA},\varOmega)
\le  \rho^{(\infty)}_{n}({\setA},\varOmega)
\le \rho^{(0)}_{n}({\setA},\varOmega)
\le
\alpha^{1/n} \left(\rho_{n-1}({\setA},\varOmega)\right)^{(n-1)/n}
\]
for any $n> 1$. Then, by passing to the upper limits in these last
inequalities, by Dai's Theorem we obtain the following generalization of the
definition of the Markovian joint spectral radius:
\begin{multline}\label{E-BWeqGen}
\rho({\setA},\varOmega):=
\limsup_{n\to\infty}\rho_{n}({\setA},\varOmega)=\\=
\limsup_{n\to\infty}\rho^{(\text{per})}_{n}({\setA},\varOmega)=
\limsup_{n\to\infty}\rho^{(\infty)}_{n}({\setA},\varOmega)=
\limsup_{n\to\infty}\rho^{(0)}_{n}({\setA},\varOmega).
\end{multline}
At last, by observing that quantities
$\bigl(\rho_{n}({\setA},\varOmega)\bigr)^{n}$,
$\bigl(\rho^{(\text{per})}_{n}({\setA},\varOmega)\bigr)^{n}$,
$\bigl(\rho^{(\infty)}_{n}({\setA},\varOmega)\bigr)^{n}$ and
$\bigl(\rho^{(0)}_{n}({\setA},\varOmega)\bigr)^{n}$ are sub-multiplicative in
$n$, we by the Fekete Lemma~\cite{Fekete:MZ23}, like in the case of formula
\eqref{E-JSRad}, can replace in \eqref{E-BWeqGen} all upper limits by limits:
\begin{multline*}
\rho({\setA},\varOmega):=
\lim_{n\to\infty}\rho_{n}({\setA},\varOmega)=\\=
\lim_{n\to\infty}\rho^{(\text{per})}_{n}({\setA},\varOmega)=
\lim_{n\to\infty}\rho^{(\infty)}_{n}({\setA},\varOmega)=
\lim_{n\to\infty}\rho^{(0)}_{n}({\setA},\varOmega).
\end{multline*}

\subsection{Products of matrices defined by subshifts of finite type}
In the symbolic dynamics the topological Markov chains are a particular case
of the so-called \emph{$k$-step topological Markov chains}
\cite[\S~1.9]{KatokHas:e} or the \emph{subshifts of finite type}
\cite{Kitchens98}. Here, by a $k$-step topological Markov chain or a subshift
of finite type is meant the restriction of the shift operator in the space
$\{1,2,\ldots,N\}^{\mathbb{Z}}$ (or $\{1,2,\ldots,N\}^{\mathbb{N}}$) to the
set of all the sequences $(i_{1},\ldots,i_{n},\ldots)$ in which admissibility
of appearance of the symbol $i_{n}$ is defined not by the immediately
preceding symbol $i_{n-1}$, as in the usual topological Markov shifts, but by
the sequence of $k$ preceding symbols $(i_{n-k},\ldots,i_{n-1})$.

The $k$-step topological Markov chains may be treated as usual ($1$-step)
Markov chains over the alphabet $\{1,2,\ldots,N\}^{k}$
\cite[\S~1.9]{KatokHas:e}. Therefore the described above approach of
$\varOmega$-lifts of sets of matrices is applicable for consideration of the
joint/generalized spectral radius of sets of matrices in which admissibility
of matrix products is defined in accordance with some $k$-step topological
Markov chains or subshifts of finite type.

\subsection{Another characteristics of matrix products}
It seems the approach of $\varOmega$-lifts of sets of matrices can be applied
to analyze some other Markovian analogs of formulae for computing the joint
spectral radius, e.g., for computing the Markovian joint spectral radius via
the trace of matrix products~\cite{ChenZhou:LAA00}.

At the same time the specific feature of the proposed approach is that all
the matrices from $\setA_{\varOmega}$ are degenerate, and among the products
of matrices from $\setA_{\varOmega}$ may occur zero matrices. This makes
improbable applying the proposed approach to study, for example, the
Markovian analogs of the lower joint spectral
radius~\cite{BM:JAMS02,PJB:SIAMJMAA10,Jungers:LAA12,GugProt:FCM13,BochiMor:ArXiv13}
introduced in~\cite{Gurv:LAA95}.

\subsection{Infinite sets of matrices}
In the paper we confined ourselves to consideration of only finite sets of
matrices $\setA$ although neither in the definitions of the joint or
generalized spectral radii (both usual and Markovian) nor in the classical
Berger-Wang Theorem~\cite{BerWang:LAA92} the finiteness of the set of
matrices $\setA$ is not required. In connection with this there arises a
question about possibility of applying the techniques of $\varOmega$-lifts to
infinite sets of matrices $\setA$.

The main difficulty here is that for infinite sets of matrices $\setA$ the
related matrix of admissible transitions $\varOmega$ is also infinite. Then
the matrices from $\setA_{\varOmega}$ are also become infinite, that is, more
adequately they should be treated as linear operators in some
infinite-dimensional spaces. But, when passing to linear operators in
infinite-dimensional spaces, the relationship between the joint and
generalized spectral radii becomes more
complicated~\cite{ShulTur:JFA00,ShulTur:SM02,Morris:JFA12}.

In connection with this the applicability of the techniques of
$\varOmega$-lifts for infinite sets of matrices $\setA$, as well as the
validity of Dai's Theorem in this case, remains to be an open question.

\subsection{Barabanov's norms}
A prominent tool in analysis of convergence of matrix products is the
so-called Barabanov's norm~\cite{Bar:AIT88-2:e} which, for a finite set of
matrices $\setA$ is defined to be a vector norm $\|\cdot\|$ such that for
some $\rho$ (necessary coinciding with $\rho(\setA)$) holds the identity
$\rho\|x\|\equiv \max_{i}\|A_{i}x\|$.

Known proofs of existence of the Barabanov norm, see, e.g.,
\cite{Bar:AIT88-2:e,Jungers:09,Wirth:CDC05}, essentially use the fact of
arbitrariness of appearance of different matrices in the related matrix
products. To the best of the author's knowledge, for the case of Markovian
joint spectral radius there are no analogs of the Barabanov norm. At the same
time, by using the procedure of $\varOmega$-lifts, for matrices from
$\setA_{\varOmega}$ all the products are allowed! Hence for such matrices it
is possible to define the norm of Barabanov.

To which extent this consideration might be useful for investigation of the
Markovian joint spectral radius will show the future.

\section*{Acknowledgments}
The author is indebted to Prof. Xiongping Dai for inspirational discussions
and numerous fruitful comments.

\bibliographystyle{elsarticle-num}
\bibliography{MarkovJSR}

\providecommand{\bbljan}[0]{January} \providecommand{\bblfeb}[0]{February}
  \providecommand{\bblmar}[0]{March} \providecommand{\bblapr}[0]{April}
  \providecommand{\bblmay}[0]{May} \providecommand{\bbljun}[0]{June}
  \providecommand{\bbljul}[0]{July} \providecommand{\bblaug}[0]{August}
  \providecommand{\bblsep}[0]{September} \providecommand{\bbloct}[0]{October}
  \providecommand{\bblnov}[0]{November} \providecommand{\bbldec}[0]{December}
\begin{thebibliography}{10}
\expandafter\ifx\csname url\endcsname\relax
  \def\url#1{\texttt{#1}}\fi
\expandafter\ifx\csname urlprefix\endcsname\relax\def\urlprefix{URL }\fi
\expandafter\ifx\csname href\endcsname\relax
  \def\href#1#2{#2} \def\path#1{#1}\fi

\bibitem{Dai:JFI14}
X.~Dai,
  \href{http://www.sciencedirect.com/science/article/pii/S001600321400012X}{Robust
  periodic stability implies uniform exponential stability of {M}arkovian jump
  linear systems and random linear ordinary differential equations}, J.
  Franklin Inst. 351~(5) (2014) 2910--2937.
\newblock \href {http://arxiv.org/abs/1307.4209} {\path{arXiv:1307.4209}},
  \href {http://dx.doi.org/10.1016/j.jfranklin.2014.01.010}
  {\path{doi:10.1016/j.jfranklin.2014.01.010}}.
\newline\urlprefix\url{http://www.sciencedirect.com/science/article/pii/S001600321400012X}

\bibitem{RotaStr:IM60}
G.-C. Rota, G.~Strang, A note on the joint spectral radius, Nederl. Akad.
  Wetensch. Proc. Ser. A 63 = Indag. Math. 22 (1960) 379--381.

\bibitem{ShihWP:LAA97}
M.-H. Shih, J.-W. Wu, C.-T. Pang,
  \href{http://www.sciencedirect.com/science/article/pii/0024379595005927}{Asymptotic
  stability and generalized {G}elfand spectral radius formula}, Linear Algebra
  Appl. 252 (1997) 61--70.
\newblock \href {http://dx.doi.org/10.1016/0024-3795(95)00592-7}
  {\path{doi:10.1016/0024-3795(95)00592-7}}.
\newline\urlprefix\url{http://www.sciencedirect.com/science/article/pii/0024379595005927}

\bibitem{DaubLag:LAA92}
I.~Daubechies, J.~C. Lagarias,
  \href{http://www.sciencedirect.com/science/article/pii/002437959290012Y}{Sets
  of matrices all infinite products of which converge}, Linear Algebra Appl.
  161 (1992) 227--263.
\newblock \href {http://dx.doi.org/10.1016/0024-3795(92)90012-Y}
  {\path{doi:10.1016/0024-3795(92)90012-Y}}.
\newline\urlprefix\url{http://www.sciencedirect.com/science/article/pii/002437959290012Y}

\bibitem{DaubLag:LAA01}
I.~Daubechies, J.~C. Lagarias,
  \href{http://www.sciencedirect.com/science/article/pii/S0024379500003141}{Corrigendum/addendum
  to: ``{S}ets of matrices all infinite products of which converge'' [{L}inear
  {A}lgebra {A}ppl.\ {\bf 161} (1992), 227--263; {MR}1142737 (93f:15006)]},
  Linear Algebra Appl. 327~(1-3) (2001) 69--83.
\newblock \href {http://dx.doi.org/10.1016/S0024-3795(00)00314-1}
  {\path{doi:10.1016/S0024-3795(00)00314-1}}.
\newline\urlprefix\url{http://www.sciencedirect.com/science/article/pii/S0024379500003141}

\bibitem{BerWang:LAA92}
M.~A. Berger, Y.~Wang,
  \href{http://www.sciencedirect.com/science/article/pii/002437959290267E}{Bounded
  semigroups of matrices}, Linear Algebra Appl. 166 (1992) 21--27.
\newblock \href {http://dx.doi.org/10.1016/0024-3795(92)90267-E}
  {\path{doi:10.1016/0024-3795(92)90267-E}}.
\newline\urlprefix\url{http://www.sciencedirect.com/science/article/pii/002437959290267E}

\bibitem{KatokHas:e}
A.~Katok, B.~Hasselblatt, Introduction to the modern theory of dynamical
  systems, Vol.~54 of Encyclopedia of Mathematics and its Applications,
  Cambridge University Press, Cambridge, 1995, with a supplementary chapter by
  Katok and Leonardo Mendoza.

\bibitem{Kitchens98}
B.~P. Kitchens,
  \href{http://link.springer.com/book/10.1007\%2F978-3-642-58822-8}{Symbolic
  dynamics}, Universitext, Springer-Verlag, Berlin, 1998, one-sided, two-sided
  and countable state Markov shifts.
\newblock \href {http://dx.doi.org/10.1007/978-3-642-58822-8}
  {\path{doi:10.1007/978-3-642-58822-8}}.
\newline\urlprefix\url{http://link.springer.com/book/10.1007\%2F978-3-642-58822-8}

\bibitem{Fekete:MZ23}
M.~Fekete,
  \href{http://link.springer.com/article/10.1007\%2FBF01504345}{{\"{U}}ber die
  {V}erteilung der {W}urzeln bei gewissen algebraischen {G}leichungen mit
  ganzzahligen {K}oeffizienten}, Math. Z. 17~(1) (1923) 228--249.
\newblock \href {http://dx.doi.org/10.1007/BF01504345}
  {\path{doi:10.1007/BF01504345}}.
\newline\urlprefix\url{http://link.springer.com/article/10.1007\%2FBF01504345}

\bibitem{PolyaSezgoI}
G.~P{\'o}lya, G.~Szeg{\H{o}}, Problems and theorems in analysis. {I}, Classics
  in Mathematics, Springer-Verlag, Berlin, 1998, series, integral calculus,
  theory of functions, Translated from the German by Dorothee Aeppli, Reprint
  of the 1978 English translation.

\bibitem{Els:LAA95}
L.~Elsner,
  \href{http://www.sciencedirect.com/science/article/pii/002437959300320Y}{The
  generalized spectral-radius theorem: an analytic-geometric proof}, Linear
  Algebra Appl. 220 (1995) 151--159, proceedings of the {W}orkshop
  ``{N}onnegative {M}atrices, {A}pplications and {G}eneralizations'' and the
  {E}ighth {H}aifa {M}atrix {T}heory {C}onference ({H}aifa, 1993).
\newblock \href {http://dx.doi.org/10.1016/0024-3795(93)00320-Y}
  {\path{doi:10.1016/0024-3795(93)00320-Y}}.
\newline\urlprefix\url{http://www.sciencedirect.com/science/article/pii/002437959300320Y}

\bibitem{Bochi:LAA03}
J.~Bochi,
  \href{http://www.sciencedirect.com/science/article/pii/S0024379502006584}{Inequalities
  for numerical invariants of sets of matrices}, Linear Algebra Appl. 368
  (2003) 71--81.
\newblock \href {http://arxiv.org/abs/math/0206128}
  {\path{arXiv:math/0206128}}, \href
  {http://dx.doi.org/10.1016/S0024-3795(02)00658-4}
  {\path{doi:10.1016/S0024-3795(02)00658-4}}.
\newline\urlprefix\url{http://www.sciencedirect.com/science/article/pii/S0024379502006584}

\bibitem{Dai:JMAA11}
X.~Dai, \href{http://dx.doi.org/10.1016/j.jmaa.2010.12.059}{Extremal and
  {B}arabanov semi-norms of a semigroup generated by a bounded family of
  matrices}, J. Math. Anal. Appl. 379~(2) (2011) 827--833.
\newblock \href {http://dx.doi.org/10.1016/j.jmaa.2010.12.059}
  {\path{doi:10.1016/j.jmaa.2010.12.059}}.
\newline\urlprefix\url{http://dx.doi.org/10.1016/j.jmaa.2010.12.059}

\bibitem{HJ:e}
R.~A. Horn, C.~R. Johnson, Topics in matrix analysis, Cambridge University
  Press, Cambridge, 1994, corrected reprint of the 1991 original.

\bibitem{Wirth:LAA02}
F.~Wirth,
  \href{http://www.sciencedirect.com/science/article/pii/S0024379501004463}{The
  generalized spectral radius and extremal norms}, Linear Algebra Appl. 342
  (2002) 17--40.
\newblock \href {http://dx.doi.org/10.1016/S0024-3795(01)00446-3}
  {\path{doi:10.1016/S0024-3795(01)00446-3}}.
\newline\urlprefix\url{http://www.sciencedirect.com/science/article/pii/S0024379501004463}

\bibitem{Koz:LAA10}
V.~Kozyakin,
  \href{http://www.sciencedirect.com/science/article/pii/S0024379510000418}{An
  explicit {L}ipschitz constant for the joint spectral radius}, Linear Algebra
  Appl. 433~(1) (2010) 12--18.
\newblock \href {http://arxiv.org/abs/0909.3170} {\path{arXiv:0909.3170}},
  \href {http://dx.doi.org/10.1016/j.laa.2010.01.028}
  {\path{doi:10.1016/j.laa.2010.01.028}}.
\newline\urlprefix\url{http://www.sciencedirect.com/science/article/pii/S0024379510000418}

\bibitem{ChenZhou:LAA00}
Q.~Chen, X.~Zhou,
  \href{http://www.sciencedirect.com/science/article/pii/S002437950000149X}{Characterization
  of joint spectral radius via trace}, Linear Algebra Appl. 315~(1-3) (2000)
  175--188.
\newblock \href {http://dx.doi.org/10.1016/S0024-3795(00)00149-X}
  {\path{doi:10.1016/S0024-3795(00)00149-X}}.
\newline\urlprefix\url{http://www.sciencedirect.com/science/article/pii/S002437950000149X}

\bibitem{BM:JAMS02}
T.~Bousch, J.~Mairesse,
  \href{http://www.ams.org/journals/jams/2002-15-01/S0894-0347-01-00378-2/}{Asymptotic
  height optimization for topical {IFS}, {T}etris heaps, and the finiteness
  conjecture}, J. Amer. Math. Soc. 15~(1) (2002) 77--111 (electronic).
\newblock \href {http://dx.doi.org/10.1090/S0894-0347-01-00378-2}
  {\path{doi:10.1090/S0894-0347-01-00378-2}}.
\newline\urlprefix\url{http://www.ams.org/journals/jams/2002-15-01/S0894-0347-01-00378-2/}

\bibitem{PJB:SIAMJMAA10}
V.~Y. Protasov, R.~M. Jungers, V.~D. Blondel,
  \href{http://epubs.siam.org/doi/abs/10.1137/090759896}{Joint spectral
  characteristics of matrices: a conic programming approach}, SIAM J. Matrix
  Anal. Appl. 31~(4) (2009/10) 2146--2162.
\newblock \href {http://dx.doi.org/10.1137/090759896}
  {\path{doi:10.1137/090759896}}.
\newline\urlprefix\url{http://epubs.siam.org/doi/abs/10.1137/090759896}

\bibitem{Jungers:LAA12}
R.~M. Jungers,
  \href{http://www.sciencedirect.com/science/article/pii/S0024379512002704}{On
  asymptotic properties of matrix semigroups with an invariant cone}, Linear
  Algebra Appl. 437~(5) (2012) 1205--1214.
\newblock \href {http://dx.doi.org/10.1016/j.laa.2012.04.006}
  {\path{doi:10.1016/j.laa.2012.04.006}}.
\newline\urlprefix\url{http://www.sciencedirect.com/science/article/pii/S0024379512002704}

\bibitem{GugProt:FCM13}
N.~Guglielmi, V.~Protasov,
  \href{http://link.springer.com/article/10.1007\%2Fs10208-012-9121-0}{Exact
  computation of joint spectral characteristics of linear operators}, Found.
  Comput. Math. 13~(1) (2013) 37--97.
\newblock \href {http://arxiv.org/abs/1106.3755} {\path{arXiv:1106.3755}},
  \href {http://dx.doi.org/10.1007/s10208-012-9121-0}
  {\path{doi:10.1007/s10208-012-9121-0}}.
\newline\urlprefix\url{http://link.springer.com/article/10.1007\%2Fs10208-012-9121-0}

\bibitem{BochiMor:ArXiv13}
J.~Bochi, I.~D. Morris, \href{http://arxiv.org/abs/1309.0319}{Continuity
  properties of the lower spectral radius}, ArXiv.org e-Print archive (Sep.
  2013).
\newblock \href {http://arxiv.org/abs/1309.0319} {\path{arXiv:1309.0319}}.
\newline\urlprefix\url{http://arxiv.org/abs/1309.0319}

\bibitem{Gurv:LAA95}
L.~Gurvits,
  \href{http://www.sciencedirect.com/science/article/pii/0024379595900063}{Stability
  of discrete linear inclusion}, Linear Algebra Appl. 231 (1995) 47--85.
\newblock \href {http://dx.doi.org/10.1016/0024-3795(95)90006-3}
  {\path{doi:10.1016/0024-3795(95)90006-3}}.
\newline\urlprefix\url{http://www.sciencedirect.com/science/article/pii/0024379595900063}

\bibitem{ShulTur:JFA00}
V.~S. Shulman, Y.~V. Turovski{\u\i},
  \href{http://www.sciencedirect.com/science/article/pii/S0022123600936401}{Joint
  spectral radius, operator semigroups, and a problem of {W}. {W}ojty{\'n}ski},
  J. Funct. Anal. 177~(2) (2000) 383--441.
\newblock \href {http://dx.doi.org/10.1006/jfan.2000.3640}
  {\path{doi:10.1006/jfan.2000.3640}}.
\newline\urlprefix\url{http://www.sciencedirect.com/science/article/pii/S0022123600936401}

\bibitem{ShulTur:SM02}
V.~S. Shulman, Y.~V. Turovski{\u\i},
  \href{http://journals.impan.pl/cgi-bin/doi?sm149-1-2}{Formulae for joint
  spectral radii of sets of operators}, Studia Math. 149~(1) (2002) 23--37.
\newblock \href {http://dx.doi.org/10.4064/sm149-1-2}
  {\path{doi:10.4064/sm149-1-2}}.
\newline\urlprefix\url{http://journals.impan.pl/cgi-bin/doi?sm149-1-2}

\bibitem{Morris:JFA12}
I.~D. Morris,
  \href{http://www.sciencedirect.com/science/article/pii/S002212361100365X}{The
  generalised {B}erger-{W}ang formula and the spectral radius of linear
  cocycles}, J. Funct. Anal. 262~(3) (2012) 811--824.
\newblock \href {http://arxiv.org/abs/0906.2915} {\path{arXiv:0906.2915}},
  \href {http://dx.doi.org/10.1016/j.jfa.2011.09.021}
  {\path{doi:10.1016/j.jfa.2011.09.021}}.
\newline\urlprefix\url{http://www.sciencedirect.com/science/article/pii/S002212361100365X}

\bibitem{Bar:AIT88-2:e}
N.~E. Barabanov, On the {L}yapunov exponent of discrete inclusions. {I},
  {A}utomat. {R}emote {C}ontrol 49~(2) (1988) 152--157, translation from
  {A}vtomat. i {T}elemekh. (1988), no. 2, 40--46.

\bibitem{Jungers:09}
R.~Jungers,
  \href{http://www.springerlink.com/content/l6h325u40254/\#section=70183&page=1}{The
  joint spectral radius}, Vol. 385 of Lecture Notes in Control and Information
  Sciences, Springer-Verlag, Berlin, 2009, {T}heory and applications.
\newblock \href {http://dx.doi.org/10.1007/978-3-540-95980-9}
  {\path{doi:10.1007/978-3-540-95980-9}}.
\newline\urlprefix\url{http://www.springerlink.com/content/l6h325u40254/\#section=70183&page=1}

\bibitem{Wirth:CDC05}
F.~Wirth,
  \href{http://ieeexplore.ieee.org/xpl/articleDetails.jsp?arnumber=1582624}{On
  the structure of the set of extremal norms of a linear inclusion}, in:
  Proceedings of the 44th {IEEE} Conference on Decision and Control, and the
  European Control Conference 2005 Seville, Spain, December 12--15, 2005, pp.
  3019--3024.
\newblock \href {http://dx.doi.org/10.1109/CDC.2005.1582624}
  {\path{doi:10.1109/CDC.2005.1582624}}.
\newline\urlprefix\url{http://ieeexplore.ieee.org/xpl/articleDetails.jsp?arnumber=1582624}

\end{thebibliography}

\end{document}